\documentclass[11pt]{amsart}
\usepackage{a4wide}
\usepackage{amssymb,amsmath,epsfig,mathrsfs, enumerate, xparse}
\usepackage{graphicx}
\usepackage{fancyhdr}
\usepackage{tikz}
\usepackage{ esint }
\usepackage{amsthm}
\usepackage{mathtools}
\usepackage{amsfonts}
\usepackage{amscd}
\usepackage{amsxtra}
\usepackage{dsfont}
\usepackage{esint}
\usepackage{xcolor}
\usepackage[margin=2.5cm]{geometry}
\usepackage{hyperref}

\usepackage[maxbibnames=99,giveninits=true]{biblatex} %Imports biblatex package
\theoremstyle{plain}
\newtheorem{thrm}{Theorem}[section]
\newtheorem*{thrm*}{Theorem}
\newtheorem{lemma}[thrm]{Lemma}
\newtheorem{prop}[thrm]{Proposition}
\newtheorem{cor}[thrm]{Corollary}
\newtheorem{rmrk}[thrm]{Remark}
\newtheorem{dfn}[thrm]{Definition}

\allowdisplaybreaks
\begin{document}
% begin top matter
% ***************** macroes needed for this paper ************************
\newcommand{\psr}{P^-}
\newcommand{\qr}{Q_\rho}
\newcommand{\sn}{\mathbb{S}}
\newcommand{\SL}{\mathcal L^{1,p}( D)}
\newcommand{\Lp}{L^p( Dega)}
\newcommand{\CO}{C^\infty_0( \Omega)}
\newcommand{\Rm}{\mathbb R^m}
\newcommand{\R}{\mathbb R}
\newcommand{\Om}{\Omega}
\newcommand{\Hn}{\mathbb H^n}
\newcommand{\aB}{\alpha B}
\newcommand{\eps}{\ve}
\newcommand{\BVX}{BV_X(\Omega)}
\newcommand{\p}{\partial}
\newcommand{\IO}{\int_\Omega}
\newcommand{\bG}{\boldsymbol{G}}
\newcommand{\bg}{\mathcal g}
\newcommand{\bz}{\mathcal z}
\newcommand{\bv}{\mathcal v}
\newcommand{\Bux}{\mbox{Box}}
\newcommand{\e}{\varepsilon}
\newcommand{\W}{\mathcal W}
\newcommand{\la}{\lambda}
\newcommand{\vf}{\varphi}
\newcommand{\rhh}{|\nabla_H \rho|}
\newcommand{\Ba}{\mathcal{B}_\beta}
\newcommand{\Za}{Z_\beta}
\newcommand{\ra}{\rho_\beta}
\newcommand{\na}{\nabla_\beta}
\newcommand{\vt}{\vartheta}
\def\dist{\textup{dist}}
\def\H{\mathcal{H}}
\newcommand{\okappa}{\overline\kappa}
\def\beq{\begin{equation}}
\def\eeq{\end{equation}}
\def\sd{\textnormal{sd}}
\numberwithin{equation}{section}

\newcommand{\N}{\mathbb N}
\newcommand{\Z}{\mathbb Z}

\definecolor{cadmiumgreen}{rgb}{0.3, 0.6, 0.35}

\newcommand{\T}{\mathbb T}
\newcommand{\Sob}{S^{1,p}(\Omega)}
\newcommand{\Dxk}{\frac{\partial}{\partial x_k}}
\newcommand{\Co}{C^\infty_0(\Omega)}
\newcommand{\Je}{J_\ve}
\let\div\undefined
\def\div{\textnormal{div}}
\newcommand\dani[1]{{\textcolor{blue}{#1}}}
\newcommand\anna[1]{{\textcolor{red}{#1}}}
\newcommand\ved[1]{{\textcolor{cadmiumgreen}{#1}}}
\newcommand{\bd}{\partial}
\newcommand{\ud}{\,\textnormal d}

\title[Asymptotic of area-preserving geometric flows]{The asymptotic of the Mullins-Sekerka and  the area-preserving curvature flow in the planar flat torus}

 \author[Vedansh Arya]{Vedansh Arya}
\address[Vedansh Arya]{Department of Mathematics and Statistics, University of Jyväskylä, Finland}\email{vedansh.v.arya@jyu.fi}
\author[D. De Gennaro]{Daniele De Gennaro}
\address[Daniele De Gennaro]{Department of Decision Sciences and BIDSA, Bocconi University, Milano, Italy}
\email{daniele.degennaro@unibocconi.it}
\author[A. Kubin]{Anna Kubin}
\address[Anna Kubin]{Institute of Analysis and Scientific Computing, Technische Universit\"at Wien, Vienna, Austria}
\email{anna.kubin@tuwien.ac.at}

\begin{abstract}
We study the asymptotic behavior of flat flow solutions to the periodic and planar two-phase Mullins-Sekerka flow and area-preserving curvature flow. We show that flat flows converge to either a finite union of equally sized disjoint disks or to a finite union of disjoint strips or to the complement of these configurations exponentially fast. 
A key ingredient in our approach is the derivation of a sharp quantitative Alexandrov inequality for periodic smooth sets.
\end{abstract}

\maketitle

%\tableofcontents

\section{Introduction}

In this paper, we study the asymptotic behavior and rate of convergence of the Mullins-Sekerka flow and the area-preserving curvature flow in the planar periodic setting.

We start by recalling that a smooth flow of sets $\{E(t)\}_{t \in [0,T)} \subseteq \T^2$ is a solution to the two-phase Mullins-Sekerka flow in the time interval $[0,T)$ if it satisfies for $t\in[0,T)$
\begin{equation}\label{eq:Mull-Sek}
\begin{cases}
V(x,t)= [\partial_{\nu_{E(t)}} u(x,t)] & \text{for  } x\in\bd E(t), \\
\Delta u(x,t)=0 & \text{for } x\in\T^2 \setminus \partial E(t),\\
u(x,t) = \kappa_{E(t)}(x) & \text{for } x\in \partial E(t),
\end{cases}
\end{equation}
where $V$ denotes the velocity in the  outer normal direction $\nu_{E(t)}$, $\kappa_{E(t)}$ is the curvature of $E(t)$, $\Delta u$ is the Laplacian of $u$ with respect to the $x$-components,
$[\partial_{\nu_{E(t)}} u]$ denotes the jump of the normal derivative of $u$ at $\p E(t)$, i.e., $[\partial_{\nu_{E(t)}}u]:= \p_{\nu_{E(t)}}u^+ - \p_{\nu_{E(t)}}u^-,$ with $u^+$ and $u^-$ denoting the restrictions of $u$ to $\T^2 \setminus E(t)$ and $E(t)$ respectively.
This evolution 
has been proposed in the physical literature in~\cite{MulSek}, see also~\cite{Peg} for more details.

Along this flow, the area is preserved since an integration by parts yields
\[
\frac{\ud}{\ud t}|E(t)|=\int_{\p E(t)}V(x,t) \ud\H^1(x)=\int_{\p E(t)} [\partial_{\nu_{E(t)}}u(x,t)] \ud\H^1(x)=-\int_{\T^2 } \Delta u(x,t) \ud x =0.
\] 
Moreover, the perimeter decreases along the flow since
\begin{align*}
\frac{\ud}{\ud t}\H^1(\p E(t))&=\int_{\p E(t)}V(x,t) \kappa_{E(t)} \ud\H^1(x)=\int_{\p E(t)} [\partial_{\nu_{E(t)}}u(x,t)]u(x,t) \ud\H^1(x)\\
& =-\int_{\T^2}|\nabla u(x,t)|^2 \ud x\le 0.
\end{align*}
Furthermore, the Mullins-Sekerka flow can be seen as the gradient flow of the perimeter with respect to a suitable $H^{-\frac12}$-Riemannian structure~\cite{QLe}. Also, the Mullins-Sekerka flow can be obtained as a sharp-interface limit of the Cahn-Hilliard equation \cite{ABC, Peg}.

We will also be interested in the area-preserving   curvature flow. This flow is defined by the evolution law
\begin{align}\label{eq:mcf}
V(x,t)=-\kappa_{E(t)}+\bar{\kappa}_{E(t)} \ \quad \text{for } x\in\p E(t),
\end{align}
where  $\bar{\kappa}_{E(t)}:=\fint_{\p E(t)} \kappa_{E(t)} d \H^1$ denotes the mean over $\bd E(t)$ of the curvature of $E(t)$. 
This type of evolution has been proposed in the physical literature as a model for coarsening phenomena~\cite{Mul,Mullins}. 
Like the Mullins-Sekerka flow,  the area of the sets $E(t)$  is preserved along the flow and the perimeter is non-increasing. 
Another important feature of the curvature flow is that it can be formally seen as a 
$L^2$-gradient flow of the perimeter.

In general, smooth solutions to \eqref{eq:Mull-Sek}  and \eqref{eq:mcf} may develop singularities in finite time (see for instance~\cite{EscIto,May01sing,MaySim}). 
Taking advantage of the  gradient flow structure of the two flows considered, one can define weak solutions to \eqref{eq:Mull-Sek}  and \eqref{eq:mcf} by means of the minimizing movement scheme (introduced in this setting in~\cite{ATW,LS}). This scheme defines discrete-in-time approximations of the continuum flow, usually called \textit{discrete flows}, depending on a time parameter $h$.  The $L^1$-limit points as $h\to0$ of discrete flows  are called   \textit{flat flows}, which are thus family of sets $E(t)$ defined at each time $t\in[0,\infty).$
After the construction of this global-in-time weak solution, it is a natural problem
to investigate its asymptotics.

There is extensive literature on the asymptotic behavior of solutions to these geometric flows.
On the one hand, under various geometric assumptions on the initial datum, one is able to show the global-in-time existence of \textit{smooth} solutions to \eqref{eq:Mull-Sek} or \eqref{eq:mcf}, and characterize their asymptotic behavior. Concerning the Mullins-Sekerka flow, we cite~\cite{AFJM, Che, EscSim98, GarRau}, while some references for the volume-preserving mean curvature flow are~\cite{BelCasChaNov,CesNov23,ES,DegDiaKubKub,Nii}.

Alternatively, one can directly study \textit{discrete flows} or \textit{flat flows}, a perspective that has gained significant interest in light of recent results on \textit{weak-strong uniqueness} for the flows under consideration. Specifically, these results demonstrate that, as long as the classical solution to \eqref{eq:Mull-Sek} or \eqref{eq:mcf} exists, any flat flow coincides with it. This has been proven for \eqref{eq:Mull-Sek} in~\cite{FisHenLauSim, HenSti} (in dimension two) and for \eqref{eq:mcf} in~\cite{JN}, under certain regularity assumptions on the initial datum, see also~\cite{Lau24} for a \textit{weak-strong uniqueness} result for a different weak notion of volume-preserving mean curvature flow.
The situation in the Euclidean setting  $\R^2$ and $\R^3$ is  well understood for the mean curvature flow \eqref{eq:mcf}. The first results concern the convergence \textit{up to translations} of \textit{flat flows} toward balls, as proved in~\cite{JN2020} for $N=2,3$. Later on, thanks to a novel quantitative version of the Alexandrov Theorem with sharp exponents, in~\cite{MPS} the authors proved the convergence of \textit{discrete flows} toward balls, with exponential rates, and without the additional translations. Subsequently, they managed to extend this study to the more challenging case of \textit{flat flows} in~\cite{JulMorPonSpa, JulMorOroSpa}. See also~\cite{KimKwo} for similar results in the planar anisotropic case. Results for \textit{flat flows} solutions to \eqref{eq:Mull-Sek} in $\T^2,\T^3$ are again contained in~\cite{JulMorPonSpa,JulMorOroSpa}, under the assumption that the initial datum $E_0$ 
has perimeter under a fixed threshold.  In $\T^2$, this amounts to ask that the initial datum $E_0$ satisfies $P(E_0)<2$. This assumption is crucial.
Indeed, since the flow does not increase the perimeter, the only possible limit points of the flow are unions of balls, therefore the authors could essentially apply the stability results they obtained in $\R^2$ without much change.

We focus in this paper on the planar, periodic setting $\T^2.$ Working in the periodic setting $\T^N$  is  more challenging than considering the Euclidean setting $\R^N$, since the possible limit points for the flows \eqref{eq:Mull-Sek},\eqref{eq:mcf} include more configurations than just balls. For instance, we find strips and, in dimensions three and higher, also cylinders and more complex periodic configurations known as gyroids. In~\cite{DeKu} two of the authors focused on the study of the asymptotic behavior of \textit{discrete flows} to \eqref{eq:mcf} in $\T^N$, starting from small deformations of stable (in a suitable sense) sets. Taking advantage of the simpler structure of stable sets in the planar case, a result for general initial sets was provided in dimension two. The convergence rate is exponential as well, thanks to a quantitative  version of  Alexandrov Theorem holding for periodic configurations. See also~\cite{DegKuKu} for the asymptotic behavior of \textit{discrete flows} to the fractional counterpart of \eqref{eq:mcf}.

This paper extends the results of~\cite{DeKu, JulMorPonSpa} in two directions. We show the first general result for the   asymptotic behavior of \textit{flat flows} for the Mullins-Sekerka flow \eqref{eq:Mull-Sek} starting from any set $E$ of finite perimeter, without requiring $P(E)<2$. This completes the theory for the Mullins-Sekerka flow in the planar, periodic setting   started in~\cite{JulMorPonSpa}.  Secondly, we extend the results of~\cite{DeKu} to \textit{flat flows} for the area-preserving curvature flow \eqref{eq:mcf} in $\T^2$. Our main result reads as follows. 
\begin{thrm*}
    Let $\{E(t)\}_{t\ge 0}$ be a \textit{flat flow} for the Mullins-Sekerka  flow \eqref{eq:Mull-Sek}, or for the area-preserving curvature flow \eqref{eq:mcf}, starting from a set of finite perimeter $E(0)\subseteq \T^2.$ Then, there exists a constant $C>1$ such that for all $t>0$ it holds 
    \begin{equation*}
        | E(t)\triangle E_\infty| \le C e^{-\frac tC},
    \end{equation*} 
    where $E_{\infty}$ takes one of the following forms:
    \begin{itemize}
        \item[(i)] $E_\infty$ or $\T^2\setminus E_\infty$ is a finite union of disjoint open disks with equal areas;
        \item[(ii)] $E_\infty$ is a finite union of disjoint open strips, possibly with different areas.
    \end{itemize}
\end{thrm*}
We will show also improved convergence results under additional assumptions on the \textit{flat flow} (see Corollaries \ref{cor:convper} and \ref{cor:conveper+hausd}).
We also remark that the constant $C$ in the statement of the theorem depends on the initial set and cannot be bounded by a universal constant. This can be seen  by considering the union of two disks of different sizes as initial datum.

A key element in the proof of our main theorem is a novel stability estimate, in the form of a quantitative Alexandrov Theorem inspired by~\cite{JulMorPonSpa} and extending the one of~\cite{DeKu}. To proceed, we introduce some notation. We define a strip to be a set bounded by two parallel lines.
\begin{dfn}\label{definizione}
    Given $m>0$ and $M>0$, let $\mathcal{P}_{m,M}$ denote  the perimeters of the union of disjoint disks or the union of disjoint parallel strips in $ \T^2$ with total area $m$ and perimeter less than or equal to $M$. 
\end{dfn}
Observe that $\mathcal{P}_{m,M}$ is a finite set (see~\cite[Remark 6.2]{DeKu}). A simplified version of our stability inequality reads as follows.
\begin{thrm}\label{teo alex}
    Let $m, M>0$ and let $E \subseteq \T^2$ be a set of class $C^2$, with $|E|=m$ and $P(E)\le M$. Then there exists a constant $C(m,M)>0$ such that
    \begin{equation}\label{eq:alex1}
        \min_{P \in \mathcal{P}_{m,M}} |P(E)-P|\le C \|\kappa_E-\bar \kappa_E\|^2_{L^2(\partial E)}.
    \end{equation}
    Moreover, if  $\delta_0>0$  is such that $P\le P(E) \le P'-\delta_0$ for every $P< P'\in \mathcal{P}_{m,M}$, then it holds
    \begin{equation*}
        P(E)-P\le C\|\kappa_E-\bar \kappa_E\|^2_{L^2(\partial E)},
    \end{equation*}
    with $C=C(m,M,\delta_0).$
\end{thrm}
Note that  the right hand side of \eqref{eq:alex1}  has quadratic dependence on $\|\kappa_E-\bar \kappa_E\|^2_{L^2(\partial E)}$, which is optimal. 
We highlight here the major difficulty that one encounters when using Theorem~\ref{teo alex} in our setting, compared to~\cite{JulMorPonSpa}. It is possible that for a    given value  $P \in \mathcal{P}_{m,M}$  there is more than one possible limit configuration corresponding to $P$: consider for instance the case of a horizontal and a vertical strip. 
Thus, the bounds on   $\|\kappa_E-\bar \kappa_E\|^2_{L^2(\partial E)}$ that one can prove along the flow are not sufficient to fix (up to translations) the limiting configuration, making the arguments more involved. Furthermore, the volume of each connected component, in general, cannot be bounded from below, as can be seen considering two parallel strips with equal total volume  differently divided between the two.
We conclude by underlining another relevant difference between our result and the main one in~\cite{JulMorPonSpa}. In the case when the $L^1$-limit points of the \textit{flat flows} are strips, we fall short of proving the convergence of the perimeters along the flow without further assumptions on the initial datum. 
This is essentially due to the fact that if the limit configuration is made of strips then it is not unique (up to translations) if one fixes the perimeter, as we can change the volume of each connected component.  This is not the  case for unions of balls.  In particular, we cannot rule out that the Hausdorff limit of the flow has connected components of zero area.

\section{Quantitative Alexandrov Theorem}

This section is devoted to the proof of a quantitative version of the Alexandrov Theorem. For notation and basic definitions, we refer to~\cite{DeKu,JulMorPonSpa}. We start by recalling a preliminary result.
\begin{lemma}\label{GB}
Let $\gamma: [0,l] \rightarrow \T^2$ be a unit speed, simple, closed $C^2$ curve. Then
$$\int_0^{l} \kappa(\gamma(\tau)) \ud\tau \in \{0,\pm2\pi\}.$$
Moreover, in the case $\int_0^l \kappa(\gamma(\tau)) \ud \tau = 0$, the length of the curve $\gamma$ is at least 1.
\end{lemma}

The proof of Lemma~\ref{GB} follows reasoning as in the proof of~\cite[Theorem 1]{Rei}. One considers the lift to $\R^2$ of $\gamma([0,l])\subseteq \T^2$, and, if it bounds a region, the result follows. If not, then, by an explicit construction, one can close the loop and directly compute that $\int_0^l \kappa=0$.

We will now prove Theorem~\ref{teo alex}, which is inspired by \cite[Theorem 1.1]{JulMorPonSpa}. This will be an immediate consequence of the following result. We will denote by $D_r(x)\subseteq \R^2$ the disk of radius $r$ and centered at $x$. 
We recall that a strip is a set bounded by two parallel lines, usually denoted by 
$S$. We define the slope of a strip $S$ the angle that the normal vector $\nu_S$ forms with the x-axis, i.e., $\nu_S\cdot e_1$.

\begin{thrm}\label{prop alex}
Let $m,M>0$. Then there exist $\varepsilon>0$ and $C(m,M)>1$ with the following property. Let $E \subseteq \T^2$ be an open set of  class $C^2$ with $|E|=m$, $P(E) \le M$ and $\|\kappa_E-\bar \kappa_E\|_{L^2(\partial E)}\le \varepsilon$. Then there are  two possibilities:
\begin{itemize}
\item[(i)] $E$, or $E^c$, is diffeomorphic to a union of $d$ disjoint disks $D_1,\ldots,D_d$ at positive distance from each other, with equal radii. Moreover 
$$ \Big|P(E)-P\Big(\bigcup_{i=1}^d D_i\Big)\Big| \le C \|\kappa_E-\bar{\kappa}_E\|^2_{L^2(\partial E)};$$
\item[(ii)] $E$ is diffeomorphic to a union of $\ell $ disjoint  parallel strips $S_1,\ldots,S_\ell$ at positive distance from each other, and total mass $m$. Moreover 
$$ \Big|P(E)-P\Big(\bigcup_{i=1}^\ell S_i\Big)\Big| \le C \|\kappa_E-\bar{\kappa}_E\|^2_{L^2(\partial E)};$$
\end{itemize}
where $\ell$ is bounded from above by a constant only depending on $m$ and $M.$

Moreover, the boundary of every connected component $E_i$ of $E$ (or $E^c$) can be parametrized as a normal graph over one of the discs $D_i$, in the first case, or one of the strips $S_i$, in the second case, with $C^{1,\frac{1}{2}}$-norm of the parametrization vanishing as $\varepsilon \to 0$. Furthermore, it holds $|E_i|=|D_i|$ or $|(E^c)_i|=|D_i|$, respectively, $|E_i|=|S_i|$.
\end{thrm}

\begin{proof} 
Let $E$ be as in the statement of the theorem, and assume that $\|\kappa_E-\bar \kappa_E\|_{L^2(\partial E)}= \varepsilon_0$, for $\e_0>0$ small enough depending only on $m,M$. \\
\noindent\textbf{Step 1:} Let $\{\Gamma_i: i\in \mathcal I\}$ be the connected components of $\bd E$, where $\mathcal{I} \subseteq \N$. We claim that
\begin{align}\label{compug}
\int_{\Gamma_i}\kappa_E=\int_{\Gamma_j}\kappa_E \qquad \forall i,j \in \mathcal I.
\end{align}
By contradiction, assume that there exist $i, j \in \mathcal{I}$ with  $i \ne j$ such that $\int_{\Gamma_i}\kappa_E \ne \int_{\Gamma_j}\kappa_E$.
By Lemma~\ref{GB}, we know that $|\int_{\Gamma_i} \kappa_E|\in \{0,2\pi\}$, for every $i \in \mathcal I$, hence we distinguish three cases.
We start with the case when $\int_{\Gamma_i}\kappa_E=0$ and $\int_{\Gamma_j}\kappa_E=2\pi$. Then it holds 
\begin{align}
&\int_{\Gamma_i}|\kappa_E|^2=\int_{\Gamma_i}\left |\kappa_E-\fint_{\Gamma_i} \kappa_E\right |^2 \le \int_{\Gamma_i }\left |\kappa_E-\bar \kappa_E\right|^2\le \varepsilon_0\label{est kappa^2},\\
&\int_{\Gamma_i \cup \Gamma_j}\left |\kappa_E-\frac{2\pi}{\mathcal{H}^1(\Gamma_i)+\mathcal{H}^1(\Gamma_j)}\right |^2=\int_{\Gamma_i \cup \Gamma_j}\left |\kappa_E-\fint_{\Gamma_i \cup \Gamma_j}\kappa_E\right |^2 \le\varepsilon_0\nonumber.
\end{align}
Therefore 
\begin{align*}
\int_{\Gamma_i}\left |\frac{2\pi}{\mathcal{H}^1(\Gamma_i)+\mathcal{H}^1(\Gamma_j)}\right |^2 \le 4\varepsilon_0,
\end{align*}
from which it follows that $\mathcal{H}^1(\Gamma_i) \le C(M) \varepsilon_0 $,
and this is a contradiction for $\varepsilon_0$ sufficiently small as $\mathcal H^1(\Gamma_i)\ge 1$ by Lemma~\ref{GB}.
The case when $\int_{\Gamma_i} \kappa_E=0$ and $\int_{\Gamma_j}\kappa_E=-2\pi$ is analogous. Finally we consider the case when $\int_{\Gamma_i}\kappa_E=2\pi$ and $\int_{\Gamma_j}\kappa_E=-2\pi$. In this case, we note that
\begin{align*}
\int_{\Gamma_i} \left |\frac{2\pi}{\mathcal{H}^1(\Gamma_i)}\right |^2 &\le 2 \int_{\Gamma_i}\left |\kappa_E-\frac{2\pi}{\mathcal{H}^1(\Gamma_i)}\right |^2 + 2 \int_{\Gamma_i }|\kappa_E|^2 \\
& \le 2\int_{\Gamma_i }\left |\kappa_E-\fint_{\Gamma_i }\kappa_E\right |^2 + 2 \int_{\Gamma_i \cup \Gamma_j}\left |\kappa_E-\fint_{\Gamma_i \cup \Gamma_j} \kappa_E\right |^2\\
& \le 4 \int_{\Gamma_i \cup \Gamma_j}\left |\kappa_E-\bar \kappa_E\right|^2\le 4\varepsilon_0.
\end{align*}
Thus, we obtain that $\pi^2 \le \varepsilon_0 \mathcal{H}^1(\Gamma_i) \le M \varepsilon_0,$ which is again a contradiction for $\varepsilon_0$ small enough.

\noindent\textbf{Step 2:} Thanks to Step 1, we can divide the proof into three cases. If $\int_{\Gamma_i}\kappa_E =2\pi,$ the lifting in $\R^2$ of the curve $\Gamma_i$  is the boundary of a set, thus reasoning as in~\cite{AFJM} we conclude assertion (i). If $\int_{\Gamma_i}\kappa_E =-2\pi,$ then $\int_{\Gamma_i}\kappa_{E^c} =2\pi$ and we conclude as above. We are then left with the case $\int_{\Gamma_i}\kappa_E =0.$ 
In this case, we show that each connected component $\Gamma_i$ of $\bd E$ can be written as a small perturbation of a line. 
%We show that $\Gamma_i$ is the graph of a $C^2$ function over a bounded interval.

Let $\ell_i=\mathcal{H}^1(\Gamma_i)$.  Since $\Gamma_i$ is regular, it can be parametrized  as a unit-speed curve  $\gamma : [0,\ell_i] \rightarrow \R^2$, $s \mapsto  (x(s),y(s))$. Define $L(s)=\int_0^{s} \kappa_{E}(\gamma(\tau)) d\tau,$ and note that $L(\ell_i)=\int_{\Gamma_i}\kappa_E=0$. By \eqref{est kappa^2}, it holds
\begin{align}\label{eccoqua}
|L(s)|^2 \le s  \int_{\Gamma_i}|\kappa_{E}|^2 \le s \varepsilon_0^2.
\end{align}
Furthermore, by direct computation, we have
\begin{equation}\label{len}
\begin{split}
x'(s)&=x'(0)\operatorname{cos}(L(s))+y'(0)\operatorname{sin}(L(s)),\\
y'(s)&=y'(0)\operatorname{cos}(L(s))-x'(0)\operatorname{sin}(L(s)).
\end{split}
\end{equation}
We now lift the curve to $\R^2$ and take $n_1, n_2\in\Z$ with the smallest modulus such that 
\begin{align}\label{close}
x(0)+n_1=x(\ell_i) \hspace{3mm} \text{and} \hspace{3mm}y(0)+n_2=y(\ell_i),
\end{align}
which exist since $\Gamma_i$ has finite length in $\T^2.$
Combining \eqref{eccoqua} and \eqref{len}, by Taylor expansion, we get,  for all $s\in [0,\ell_i]$, 
\begin{align*}
\left| {x}'(s)-x'(0)\right| \le C \varepsilon_0, \qquad
\left| {y}'(s)-y'(0)\right| \le C \varepsilon_0,
\end{align*}
where the constant $C>0$ only depends on $M$.
We now integrate the inequalities above in $(0,s)$ to find
\begin{align*}
\left| {x}(s)-{x}(0)-x'(0)s\right| \le C \varepsilon_0,\qquad 
\left| {y}(s)-{y}(0)-y'(0)s\right| \le C \varepsilon_0,
\end{align*}
from which it follows that $\Gamma_i$ is parametrized by a small perturbation of a line. 
Letting $s=\ell_i$ in the above estimates, by \eqref{close}, we obtain
\begin{align*}
\left| n_1-l_ix'(0)\right| \le C \varepsilon_0, \qquad
\left| n_2-l_iy'(0)\right| \le C \varepsilon_0.
\end{align*}
Hence we conclude that there exists $\sigma : [0,\ell_i] \rightarrow \R^2$ such that
\begin{align*}
{x}(s)&=x(0)+\frac{n_1}{\ell_i}s+ \sigma_1(s), \hspace{4mm}
{y}(s)=y(0)+\frac{n_2}{\ell_i}s+ \sigma_2(s),
\end{align*}
{and moreover} 
\begin{equation}\label{eq:bds}
\|\sigma(s)\|_{\infty}+\|\sigma'(s)\|_{\infty} \le C \varepsilon_0.
\end{equation}
By the implicit function theorem (see also~\cite[Lemma 3.4]{KimKwo}), the curve $\Gamma_i$ can be parametrized as a graph over the line $L_i$ with slope $n_1/n_2$ (where $L_i$ is a line parallel to $x=0$ when $n_2=0$)  and passing through $(x(0),y(0))$ with height function $f_i \in C^{1,\frac{1}{2}}(L_i)$ with 
\begin{align}\label{eq:para}
\|f_i\|_{C^1(L_i)} \le C \varepsilon_0.
\end{align}
Since $\Gamma_i$ does not bound, there exists a unique $\tilde{\Gamma}_i$ such that $\Gamma_i \cup \tilde{\Gamma}_i$ is the boundary of a connected component $E_i$ of $E.$ By the previous arguments, also $\tilde{\Gamma}_i$ is a small $C^{1,\frac12}$-deformation of a line $\tilde{L}_i$ parallel to $L_i$. By parallel shifting  the curves $\Gamma_i$ and $\tilde{\Gamma}_i$, we can find a strip $S_i$  with boundary $L_i \cup \tilde{L}_i$ and $|E_i|=|S_i|$. 
%In view of \eqref{eq:para} and the fact that $P(E) \le M,$ we can do this and still have 
Thanks to \eqref{eq:bds}, 
\eqref{eq:para} still holds, possibly increasing $C.$
By the area formula, it is easy to see that (see e.g.~\cite[Lemma 3.1]{DeKu})
$$0\le P(E_i)-P(S_i) \le C\|f_i\|_{C^1(L_i)}^2 \le C \varepsilon_0^2=C\|\kappa_E - \bar \kappa_E\|^2_{L^2(\bd E)}.$$
We can then sum over all the connected components (whose number is bounded by $M/2$) to get the desired result. This completes the proof.

\end{proof}

\section{Asymptotic Behavior of the Mullins-Sekerka Flow}

In this section we apply the quantitative Alexandrov Theorem (Theorem~\ref{teo alex}) to show the asymptotic of the Mullins-Sekerka flow.
We start by recalling the definition of \textit{flat flow} solutions of the Mullins-Sekerka flow, as given in~\cite{LS,Rog}.

Let $E\subseteq \T^2$ be  a measurable set with $|E|=m$, and consider the minimization problem
\begin{equation}
\label{minimumpb}
    \min \left\lbrace P(F)+\frac{h}{2} \int_{\T^2} |\nabla U_{F,E}|^2 \ud x : |F|=|E|\right\rbrace,
\end{equation}
where $U_{F,E}\in H^1(\T^2)$ is the solution to
\begin{equation}\label{eq:potential1}
    -\Delta U_{F,E} =\frac{1}{h}\left(\chi_F-\chi_E\right)\quad \text{in } \T^2
\end{equation}
with zero average. Let us set
 \beq
\label{def:distance2}
\mathcal{D}(F,E) := \int_{\T^2} |\nabla U_{F,E}|^2  \ud x .
\eeq
By the results of~\cite{LS, Rog}, there exists a minimizer of \eqref{minimumpb}, which may not be unique.  Note also that, defining  the $H^{-1}$-norm by duality  as 
\[
\|f\|_{H^{-1}(\T^2)} :=  \sup \Big\{ \int_{\T^2} \varphi  f \ud x : \|\nabla \varphi\|_{L^2(\T^2)} \leq 1\Big\},
\]
pairing this definition with  \eqref{eq:potential1}  and integrating by parts, we get
\beq
\label{def:Hmenouno-bound}
\|\chi_F - \chi_E\|_{H^{-1}(\T^2)}^2  \leq  h^2 \, \| \nabla U_{F,E} \|_{L^2(\T^2)}^2 = h^2 \, \mathcal{D}(F,E).
\eeq

A \textit{discrete flow} to the Mullins-Sekerka flow can be defined by iteratively minimizing \eqref{minimumpb}.  Fix $h>0$, for every measurable open set $E(0)\subseteq \T^2$ we set $E_0^{(h)}= E(0)$ and 
\[  E_{k+1}^{(h)}\in \text{argmin}\left\lbrace P(F)+\frac{h}{2} \mathcal{D}(F,E_k^{(h)}) : |F|=m\right\rbrace. \]
A \textit{discrete flow}  $\{E^{(h)}(t)\}_{t \geq 0}$ is defined by setting
\[
E^{(h)}(t) = E_k^{(h)} \qquad \text{for }\, t \in [kh, (k+1) h),
\]
and a \textit{flat flow solution for the Mullins-Sekerka flow}  \eqref{eq:Mull-Sek}   is defined as any $L^1$-cluster point  $\{E(t)\}_{t \geq 0}$ of $\{E^{(h)}(t)\}_{t \geq 0}$ as $h\to 0$, that is 
\[
E^{(h_n)}(t) \to E(t) \quad \text{ in $L^1$ for a.e }\, t >0 \text{ and for some }h_n\to 0.
\]

Let $U_{k+1}^{(h)}$ denote the potential associated to $E_{k+1}^{(h)}$ as defined in \eqref{eq:potential1}.
Note that the sets $E_{k}^{(h)}$ are $C^{3,\alpha}$-regular, as noted in~\cite{JulMorPonSpa},
and they satisfy the Euler-Lagrange equation associated to \eqref{minimumpb}, namely,  
\beq \label{eq:Euler-Lagr2}
U_{k+1}^{(h)} = -\kappa_{E_{k+1}^{(h)}} + \lambda_{k+1}^{(h)} \qquad \text{on }\, \bd E_{k+1}^{(h)}.
\eeq
Testing the minimality of $E_{k+1}^{(h)}$ with $E_{k}^{(h)}$, one obtains
\beq \label{eq:energy-compa2}
P(E_{k+1}^{(h)})  +\frac{h}{2}  \mathcal{D}(E_{k+1}^{(h)}, E_{k}^{(h)}) \leq P(E_{k}^{(h)}).
\eeq

Following the proof of~\cite[Proposition 3.1]{Rog}, we  conclude the existence of a \textit{flat flow} with initial datum $E(0)$, satisfying $P(E(t)) \leq P(E(0))$ and $|E(t)| = |E(0)|$ for every $t \geq 0$.
Moreover, the family  $\{E(t)\}_{t\geq 0}$ satisfies the equation \eqref{eq:Mull-Sek} in a weak sense.

Before proving our main result, we need some preliminary results.  We use the following notation: for $r>0$, the set $E_r$ denotes the $r$-neighbourhood  of $E$, i.e., $E_r:=\{x\in\T^2 : \text{dist}(x,E)\le r\}$.

\begin{lemma}\label{lemma Max}
    Let $0<m <1$ and $M>0$. For every $0<\eta<1-m$, there exists $r_0=r_0(\eta,m,M)$ with the following property: 
    for every connected set $E\subseteq \T^2$ of class $C^3$ with $|E|= m$ and $P(E)\le M$,  we have   $|E_r|<1-\eta$ for all $r\le r_0$.
\end{lemma}

\begin{proof}
    Let $E\subseteq \T^2$ be a connected set of class $C^3$ with $|E|= m$ and $P(E)\le M$. For given $0< \eta <1-m,$ set $r_0=\min\{\frac{1-m-\eta}{8\pi M}, \frac{1}{8}\}>0.$ Fix $r \le r_0.$ If there exists $x\in\bd E$ such that $E \subseteq D_r(x),$ then there is nothing to prove. Otherwise, we note that by compactness of $\partial E$ in $\T^2$, there exists a finite set of points $\{x_1, \dots ,x_k\}\subseteq \bd E$ such that
    \begin{equation*}
        |x_i-x_j|\ge r, \qquad \text{for every } i\ne j,
    \end{equation*}
    and  
    \begin{equation}\label{cover}
        \bd E\subseteq  \bigcup_{i=1}^k D_{r}(x_i).
    \end{equation}
    Since the balls $D_{\frac{r}{2}}(x_i)$ are disjoint, we have
    \begin{equation}\label{per1}
         P(E) \ge \sum_{i=1}^k P(E; D_{\frac r2}(x_i)).
    \end{equation}
    Moreover, since $ E$ is path-connected, for every $x\in\bd E$, it holds
    \[ P(E; D_{\frac r2}(x))\ge r, \]
    therefore \eqref{per1} entails 
    \begin{equation}\label{bound k}
        k \le \frac M{r}.
    \end{equation}
    By definition the $r$-neighbourhood of $E$ is the set $ E_r=E\cup \bigcup_{x\in\bd E} D_r(x)$,    thus by \eqref{cover} we deduce
    \[ E_r\subseteq E \cup \bigcup_{i=1}^k D_{2r}(x_i). \]
    Finally, the inclusion above, coupled with \eqref{bound k}, yields 
    \[ |E_r|\le m + 4k\pi r^2\le  m + 4 M\pi r,  \]
    from which the conclusion follows.
\end{proof}

The second lemma we need is in the spirit of~\cite[Lemma 2.1]{RogSch}.

\begin{lemma}\label{lemma4.1JMPS}
    Given $0<m<1$, $M>0$, let $E\subseteq \T^2$ be an open set of class $C^3$, with $|E|=m$ and $P(E)\le M$, and let $u_E \in C^1(\T^2)$ be a function with zero average such that $\|\nabla u_E\|_{L^2(\T^2)}\le M$ and 
    \begin{equation}\label{elems}
        \kappa_E=-u_E+\lambda \,\,\, \textnormal{on} \,\,\, \bd E \,\,\, \textnormal{for some}\,\,\, \lambda \in \R.
    \end{equation}
    Then it holds 
    \begin{equation}\label{4.9JMPS}
        \sup_{x \in \T^2,\rho >0} \frac{\mathcal{H}^1(\bd E\cap D_{\rho}(x))}{\rho}\le K,
    \end{equation}
    where $K>0$ depends only on $m$ and $M.$
\end{lemma}

\begin{proof}
    By \eqref{elems}, for every $X \in C^1(\T^2,\R^2)$ it holds
    \begin{equation}\label{eq EL}
        \int_{\bd E} \div_{\tau} X \ud \mathcal{H}^1=\int_{\bd E} \kappa_E X \cdot \nu_E \ud \mathcal{H}^1= \int_E \div((-u_E+\lambda)X) \ud x.
    \end{equation}
 Here, $\div_{\tau} X$ denotes the tangential divergence of $X$ on $\partial E$ and $\nu_E$ denotes the outward unit normal vector.  In view of~\cite[Lemma 2.1]{Sch}, to prove the lemma, it suffices to bound the Lagrange multiplier $\lambda.$
    We suppose that $E$ is connected, as the number of connected components of $E$ is uniformly bounded by a constant only depending on $m$ and $M$ and the following arguments can be repeated on each component. By Lemma~\ref{lemma Max}, for $r=\min\{\frac 18, \frac{1-m}{16\pi M}\}>0$ we have $|E_{r}|<1-\frac{1-m}2$. Consider 
    \[ f= \rho_{\frac r2}\ast ( \chi_{E_r}-c_0 \chi_{\T^2\setminus E_r}),  \]
    where $c_0= \frac{|E_r|}{1-|E_r|}  \in (0,\infty)$ so that $f$ has zero average in $\T^2$, and   $\rho_{\frac r2}$ is the rescaled standard mollifier.
    We then define $v$ as the solution with zero average to 
    \begin{equation}\label{def u}
        \Delta v=  f \qquad \text{in }\T^2.
    \end{equation}
    By standard Schauder estimates (see for instance~\cite{GiaMar}), it holds  
    \[ \|   v\|_{C^{2,\alpha}(\T^2)}\le C ( \|\nabla v\|_{L^2(\T^2)} +\| f\|_{C^\alpha(\T^2)})\le C(m,M) \]
    provided we show 
    \[ \|\nabla v\|_{L^2(\T^2)} \le  C(m,M).\]
    This  can be done by integrating by parts and using the definition of $v$
    \[  \|\nabla v\|_{L^2(\T^2)}^2=-\int_{\T^2} \Delta v \, v=\int_{\T^2}  v\, f\le C(m,M)\|v\|_{L^2(\T^2)}\le C(m,M)  \|\nabla v\|_{L^2(\T^2)},   \]
    where the last inequality follows from Poincaré inequality.
    
    We then use \eqref{eq EL} with the choice  $X=\nabla v$ to get
    \begin{equation*}
        \int_{\partial E} \div_{\tau} X \ud \mathcal{H}^1 = \int_E \div ((-u_E+\lambda)X) \ud x = -\int_E \div (u_E X) \ud x + \lambda |E|.
    \end{equation*}
  Here, we have used that $\div X= \Delta v=f = 1$ in $E.$ 
  Noting that 
    \begin{align*}
        &\left |\int_{\partial E} \div_{\tau} X \ud \mathcal{H}^1 \right |\le P(E) \|\nabla X\|_{C^0(\T^2)} \le C(m,M) P(E),\\
        &\left |  \int_{ E} \div \left( u_E X \right) \ud x \right | \le \|u_E\|_{H^1(E)} \| X\|_{C^1(\T^2)} \le  C(m,M) \|u_E\|_{H^1(E)},
    \end{align*}
    {we conclude that 
    \begin{equation}\label{bound lambda}
        |\lambda\|E|\le C(m,M)\left( P(E) +\|u_E\|_{H^1(E)} \right),
    \end{equation}
    which, in turn, gives a uniform bound on the curvature by the Euler-Lagrange equation \eqref{eq EL}. We conclude by applying~\cite[Lemma 2.1]{Sch}.}
\end{proof}

%%%%%%%%%%%%%%%%%%%%%%%%%%%%%%%%%%%%%%%%%%

We recall that, if  $E$ satisfies \eqref{4.9JMPS}, then by~\cite[Theorem 4.7]{MZ} it holds
\begin{equation}\label{4.12JMPS}
    \left | \int_{\partial E}\varphi \ud \mathcal{H}^1 \right | \le C \|\varphi\|_{H^{1}(\T^2)}, \quad \forall \varphi \in C^1(\T^2),
\end{equation}
where $C$ depends only on the constant $K$ of the previous lemma, and therefore on $m$ and $M$. 
This trace estimate is crucial for establishing the version of Theorem~\ref{prop alex} that we will apply in our setting.

\begin{prop} \label{prop:2DAle2}
 Given $0<m<1$, $M>0$, let $E\subset\T^2$ be an open set of class $C^3$, with $|E|=m$ and $P(E) \le M$, and let $u_E \in C^1(\T^2)$ be a function with zero average such that 
\[
\kappa_E = -u_E +\lambda \quad \text{on } \,  \bd E
\]
 for some $\lambda \in \R$. Then, there exist $\e_0=\e_0(m,M)\in (0,1)$ such that  if 
$\|\nabla u_E\|_{L^2(\T^2)} \leq \e_0$
then the set $E$ satisfies the conclusion of Theorem~\ref{prop alex}.
\end{prop}

\begin{proof} 
By Lemma~\ref{lemma4.1JMPS}, we can apply \eqref{4.12JMPS} with $\varphi = u_E^2 $ and obtain
\[
 \int_{\bd E} u_E^2\, \ud \H^1 \leq C \|u_E^2\|_{H^{1}(\T^2)} \leq C \|u_E\|_{H^1(\T^2)}^2 \leq C\| \nabla u_E\|_{L^2(\T^2)}^2, 
\]
where the last inequality follows from Poincar\'e inequality as $u_E$ is a function with zero average. By the assumption on $u_E$, if also $\|\nabla u_E\|_{L^2(\T^2)} \leq \e_0$, then it follows that 
\[
 \int_{\bd E} |\kappa_E - \okappa_E|^2 \, \ud \H^1\leq   \int_{\bd E} |\kappa_E - \lambda|^2 \, \ud \H^1 =  \int_{\bd E} u_E^2\, \ud \H^1\leq C\| \nabla u_E\|_{L^2(\T^2)}^2 \leq C \e_0^2. 
\]
Hence $E$ satisfies the hypothesis of Theorem~\ref{prop alex}. 
\end{proof}

Proposition~\ref{prop:2DAle2} implies the following corollary.

\begin{cor}\label{coro:2Dale2}
 Given $0<m<1$, $M>0$, let $E\subset\T^2$ be a set of class $C^3$, with $|E|=m$ and  $P(E)\le M$, and let $u_E \in C^1(\T^2)$ be a function with zero average such that $\kappa_E = -u_E +\lambda$ on $\bd E$ for some $\lambda \in \R$.  If $\delta_0>0$ and $P \in \mathcal{P}_{m,M}$ are such that $P \leq P(E) \leq P' - \delta_0$ for every $P < P' \in \mathcal{P}_{m,M}$, then it holds 
\[
P(E) - P  \leq C \|\nabla u_E\|_{L^2(\T^2)}^2, 
\]
for $C= C(m,M,\delta_0)$.
\end{cor}

We are now able to prove our main result concerning the Mullins-Sekerka flow, the proof follows the lines of~\cite[Theorem 1.3]{JulMorPonSpa}.

\begin{thrm}\label{main theorem 2}
    Let $0<m<1$ and let $\{E(t)\}_{t\ge 0}$ be a flat flow for the Mullins-Sekerka flow starting from a set of finite perimeter $E(0)\subseteq \T^2$ with $|E(0)|=m$. Then, the family $\{E(t)\}_{t\ge 0}$ has a unique $L^1$-limit point $E_\infty$, which takes one of the following forms:
    \begin{itemize}
        \item[(i)] $E_\infty$ is the union of $d\in\N$ disjoint open disks $D_r(x_1),\dots,D_r(x_d)$ with the same area and with $\pi r^2d=m$;
        \item[(ii)] $\T^2\setminus E_\infty$ is the union of $d\in\N$  disjoint open disks $D_r(x_1),\dots,D_r(x_d)$ with the same area and with $\pi r^2d=1-m$;
        \item[(iii)] $E_\infty$ is the union of $\ell \in\N$ disjoint parallel open strips $S_1,\dots, S_\ell$, possibly having different areas,  with $|\bigcup_{i=1}^\ell S_i|=m$.
    \end{itemize}
     In all cases the rate of convergence is exponential, i.e.,
        \begin{equation*}
        | E(t)\triangle E_\infty| \le C e^{-\frac tC},
    \end{equation*}
  where $C$ denotes a positive constant independent of $t$. 
\end{thrm}

\begin{proof}
Let $\{E(t)\}_{t\geq0}$ be a flat flow for the Mullins-Sekerka flow, and let $h_n \to 0$ and $\{E^{(h_n)}(t) \}_{t\geq0}$ be a discrete flow converging to $E(t)$ in $L^1$ as $n \to \infty$. 
We begin by proving that $E(t) \to E_{\infty}$ in $L^1,$ and then provide a characterization of the possible limiting configurations $E_\infty.$

From the interpolation inequality~\cite[Lemma 4.4]{JulMorPonSpa},  for every $\varepsilon >0$ and $t<s$, we have   
\[
\begin{split}
\|\chi_{E^{(h_n)}(s)} &- \chi_{E^{(h_n)}(t)}\|_{L^1(\T^2)} \\
&\leq C \e  \, \|\chi_{E^{(h_n)}(s)} - \chi_{E^{(h_n)}(t)}\|_{\text{BV}(\T^2)} + C \e^{-1} \|\chi_{E^{(h_n)}(s)} - \chi_{E^{(h_n)}(t)}\|_{H^{-1}(\T^2)}, 
\end{split}
\]
Notice that for all $t >0$, by testing the minimality of $E^{(h)}_k$ with $E^{(h)}_{k-1}$, it follows $\|\chi_{E^{(h_n)}(t)}\|_{\text{BV}(\T^2)}\leq \|\chi_{E(0)}\|_{\text{BV}(\T^2)}\le C$. Consequently, the above inequality becomes
\beq \label{eq:thm311}
\|\chi_{E^{(h_n)}(s)} - \chi_{E^{(h_n)}(t)}\|_{L^1(\T^2)} 
\leq C \e   + C \e^{-1} \|\chi_{E^{(h_n)}(s)} - \chi_{E^{(h_n)}(t)}\|_{H^{-1}(\T^2)}, 
\eeq
Using \eqref{def:Hmenouno-bound} and Jensen inequality, we find
\begin{equation*}
\begin{split}
&\|\chi_{E^{(h_n)}(s)} - \chi_{E^{(h_n)}(t)}\|_{H^{-1}(\T^2)} \le \sum_{k=\lfloor\frac{t}{h_n}\rfloor+1}^{\lfloor\frac{s}{h_n}\rfloor} \|\chi_{E^{(h_n)}_{k}} - \chi_{E^{(h_n)}_{k-1}}\|_{H^{-1}(\T^2)}  \\ 
 &\leq 
 \sum_{k=\lfloor\frac{t}{h_n}\rfloor+1}^{\lfloor\frac{s}{h_n}\rfloor}\left(h_n^2 \mathcal{D}(E^{(h_n)}_{k},E^{(h_n)}_{k-1}) \right)^{\frac{1}{2}}\leq  \frac{\sqrt{s-t}}{\sqrt{h_n}} \Big( \sum_{k=\lfloor\frac{t}{h_n}\rfloor+1}^{\lfloor\frac{s}{h_n}\rfloor} h_n^2 \, \mathcal{D}(E^{(h_n)}_{k},E^{(h_n)}_{k-1})\Big)^{\frac12}.
\end{split}
\end{equation*}
By combining the above inequality with \eqref{eq:thm311} we get
\beq \label{eq:thm312}
\|\chi_{E^{(h_n)}(s)} - \chi_{E^{(h_n)}(t)}\|_{L^1(\T^2)} 
\leq C \e   + C \e^{-1} \frac{\sqrt{s-t}}{\sqrt{h_n}} \Big( \sum_{k=\lfloor\frac{t}{h_n}\rfloor+1}^{\lfloor\frac{s}{h_n}\rfloor} h_n^2 \, \mathcal{D}(E^{(h_n)}_{k},E^{(h_n)}_{k-1})\Big)^{\frac12}. 
\eeq
Therefore, in view of \eqref{eq:thm312}, to establish the exponential convergence in $L^1$ of $\{\chi_{E(t)}\}_{t \ge 0}$ as $t \to \infty$, it is sufficient to prove the exponential decay of the last term on the right hand side of \eqref{eq:thm312}. To achieve this, we begin by summing \eqref{eq:energy-compa2}, which gives
\beq\label{eq:thm322}
\frac{h_n}{2}\sum_{k=\lfloor\frac{t}{h_n}\rfloor+1}^{\lfloor\frac{s}{h_n}\rfloor} \mathcal{D}(E^{(h_n)}_{k},E^{(h_n)}_{k-1}) \leq P(E_{\lfloor\frac{t}{h_n}\rfloor}^{(h_n)}) - P\big(E^{(h_n)}_{\lfloor\frac{s}{h_n}\rfloor}\big).
\eeq
We now define the sequence of functions $f_n:[0, \infty) \to \R$ as follows: 
\[
f_n(t) = P(E^{(h_n)}(t)).
\]
From \eqref{eq:energy-compa2}, we notice that each $f_n$ is monotone non-increasing. Therefore, up to a subsequence, (still denoted by $f_n$) $f_n$ converges pointwise as $n \to \infty$ to a non-increasing function $f_\infty:[0,\infty)\to \R$.
Set $F_\infty= \lim_{t\to \infty}f_\infty(t)$.  We divide the proof in two cases.

\textbf{Case 1}: There exist
two consecutive elements $P < P'$ of the increasing sequence $\mathcal{P}_{m,M}$ (recall Definition~\ref{definizione}) such that either $P<F_\infty<P'$, or $F_\infty=P$ and $f_\infty(t) >P$ for every $t\in [0,\infty)$.
In this case, there exists $\bar t \geq 1$ with the following property: for every $T>\bar t$ there exists $\bar n\in \N\setminus\{0\}$ such that
\begin{gather}\label{eq:thm3-1}
P\leq f_n(t) < P'\qquad \text{and}\qquad P' -f_n(t) \geq \frac{P'-F_\infty}{2}=:\delta_0
\end{gather}
for every $n\geq \bar n$ and   $t \in [ \bar t, T]$. In the following, we denote by $C$ a positive constant depending on $m$, $M$, $\delta_0$ and the flat flow itself (but not on $h$), that may change from line to line and possibly  within the same line.
From \eqref{eq:thm322} and \eqref{eq:thm3-1}, we deduce  
for every $i \in \big\{\lfloor\frac{\bar t}{h_n}\rfloor, \ldots, \lfloor\frac{T}{h_n}\rfloor\big\}$ that 
\beq\label{eq:thm3-2}
\frac{h_n}{2}\sum_{k=i+1}^{\lfloor\frac{T}{h_n}\rfloor} \mathcal{D}(E^{(h_n)}_{k},E^{(h_n)}_{k-1}) \leq P(E_i^{(h_n)}) - P\big(E^{(h_n)}_{\lfloor\frac{T}{h_n}\rfloor}\big) \leq P(E_i^{(h_n)})  -P.
\eeq
Since, by \eqref{eq:thm3-1}, we have $P \le P(E_i^{(h_n)}) \le P'-\delta_0,$ Corollary~\ref{coro:2Dale2} yields
$$P(E_i^{(h_n)}) - P \leq  C\|\nabla U_i^{(h_n)}\|_{L^2(\T^2)}^2.$$ Consequently, by \eqref{eq:thm3-2}  it holds
\[
\frac{h_n}{2}\sum_{k=i+1}^{\lfloor\frac{T}{h_n}\rfloor}  \mathcal{D}(E^{(h_n)}_{k},E^{(h_n)}_{k-1}) \leq P(E_i^{(h_n)}) - P \leq  C\|\nabla U_i^{(h_n)}\|_{L^2(\T^2)}^2= C \mathcal{D}(E^{(h_n)}_{i},E^{(h_n)}_{i-1}).
\]
Therefore we obtain 
\[
\sum_{k=i+1}^{\lfloor\frac{T}{h_n}\rfloor} \mathcal{D}(E^{(h_n)}_{k},E^{(h_n)}_{k-1}) \leq \frac{C}{h_n}\mathcal{D}(E^{(h_n)}_{i},E^{(h_n)}_{i-1}) .
\]
By~\cite[Lemma 3.1]{JulMorPonSpa}, the inequality above implies for every $t\in [\bar t, T]$ that
\beq\label{decay D}
\sum_{k=\lfloor\frac{t}{h_n}\rfloor+1}^{\lfloor\frac{T}{h_n}\rfloor}
h_n\, \mathcal{D}(E^{(h_n)}_{k},E^{(h_n)}_{k-1})
\leq  2 \left(1-\frac{h_n}{C}\right)^{\lfloor\frac{t}{h_n}\rfloor - \frac{\bar t}{h_n}}\leq  C e^{-\frac{t}{C}}
\eeq
for $h_n\leq h_0(T)$. Combining  \eqref{eq:thm312} and \eqref{decay D}, we find that for $\bar t \leq t <s \leq T$ with $s \leq t+1$ the following holds
\begin{equation*}
\begin{split}
\|\chi_{E^{(h_n)}(s)} &- \chi_{E^{(h_n)}(t)}\|_{L^1(\T^2)} \leq C \varepsilon + C \e^{-1} e^{-\frac{t}{C}},
\end{split}
\end{equation*}
when $h_n\leq h_0(T)$. Choose $\e =  e^{-\frac{t}{2C}}$ above to get
\begin{equation}\label{eq:thm3-3}
\|\chi_{E^{(h_n)}(s)} - \chi_{E^{(h_n)}(t)} \|_{L^1(\T^2)}  \leq C e^{-\frac{t}{C}}.
\end{equation}
Finally, we let $h_n \to 0$ to conclude that there exists a positive constant $C>1$ such that the flat flow satisfies 
\beq\label{end half case 1}
|E(s) \Delta E(t)|   \leq C \,  e^{-\frac{t}{C}}.
\eeq
Hence, we deduce that $E(t)$ converges to a set of finite perimeter $E_\infty$ in $L^1$ exponentially fast. 
%%%%%%%%%%

%%%%%%%%%%%%%%%%%%%%
We now show that the  limiting set $E_\infty$ is the union of disjoint open disks with the same radius or a union of strips. 
Fix $t\geq \bar t$ and $\alpha>0$ small, to be chosen later. Recall that $U^{h_n}_i$ is a function with zero average. By the trace inequality \eqref{4.12JMPS} and by Poincar\'e inequality, we obtain  
\beq\label{expo21}
\begin{split}
\int_{[t, t+e^{-\alpha t}]} \Big(\int_{\bd E^{(h_n)}(s)} |U^{(h_n)}_s|^2 d\H^1\Big) ds &\leq
h_n\sum_{i=\lfloor \frac{t}{h_n}\rfloor}^{\lfloor \frac{t+e^{-\alpha t}}{h_n}\rfloor}\int_{\T^2} |\nabla U_{i}^{(h_n)}|^2 \, d \H^1 \\
&\le C
\sum_{i=\lfloor \frac{t}{h_n}\rfloor}^{\lfloor \frac{t+e^{-\alpha t}}{h_n}\rfloor} h_n \mathcal{D}(E^{(h_n)}_{i},E^{(h_n)}_{i-1}) 
\le C   e^{-\frac{t}{C}}\,,
\end{split}
\eeq
for $n$ sufficiently large, where the last two inequalities follow, respectively, from \eqref{def:distance2} and \eqref{decay D}.
Then, by \eqref{expo21}, the mean value theorem and the Euler-Lagrange equation \eqref{eq:Euler-Lagr2}, we deduce that there exists $s_n \in [t, t+e^{-\alpha t}]$ such that 
\beq\label{expo3}
\|\kappa_{E^{(h_n)}(s_n)}-\okappa_{E^{(h_n)}(s_n)} \|_{L^2(\bd E^{(h_n)}(s_n))}^2\leq \int_{\bd E^{(h_n)}(s_n)} |U^{(h_n)}_{s_n}|^2 d\H^1 \leq Ce^{-\big(\frac{1}{C}-\alpha\big)t}\le Ce^{-\frac tC}
\eeq
whenever $\alpha>0$ is small (depending on $m,M,\delta_0$ and possibly on the flat flow).
Thus, for  $t$ sufficiently large Theorems~\ref{teo alex} and \ref{prop alex} imply that $E^{(h_n)}(s_n)$ is diffeomorphic to $F^{(h_n)}(s_n)$, where 
\begin{itemize}
\item[(a1)] $F^{(h_n)}(s_n)$ is a union of disjoint disks at positive distance from each other  with equal radii and having total perimeter $P$ and area $m$; 
\item[(a2)] $F^{(h_n)}(s_n)$ is the complement of a union of disjoint disks at positive distance from each other  with equal radii and having total perimeter $P$ and area $1-m$;  
\item[(b)] or $F^{(h_n)}(s_n)$ is a union of disjoint strips (whose area may vary in $n$) with perimeter $ P$ and area $m,$
\end{itemize} 
and furthermore it holds 
\begin{equation}\label{eq:P_est_discr}
    P(E^{(h_n)}(s_n))-P\le Ce^{-\frac{t}{C}}.
\end{equation}
We note that, in cases (a1), (a2), the limit configuration is uniquely determined (up to translations), and in case (b), it is uniquely determined up to translations and a change in the mass of the single strips.
Therefore, up to a subsequence, we can assume that all sets $E^{(h_n)}(s_n)$ are parametrized over the same configuration. 
We also observe that, in both the cases, the $C^{1,\frac 12}$-norm of the parametrization decays exponentially and, in particular, we have
\begin{equation}\label{eq:expo-discr}
    |E^{(h_n)}(s_n)\triangle F^{(h_n)}(s_n)|\le C e^{-\frac tC},
\end{equation}
for $ n \ge n_0(t)\in\N$, and where the constant $C>0$ does not depend on $t$ and $n$.
Note that for $m\ge n\in\N$ it also holds 
\begin{equation}\label{eq:expo-discre2}
    |F^{(h_m)}(s_m)\triangle F^{(h_n)}(s_n)|\le  |E^{(h_m)}(s_m)\triangle E^{(h_n)}(s_n)| + 2C e^{-\frac tC} 
\end{equation}
for $m\ge n \ge n_0(t)\in\N$. Now, passing to the limit as $h_n\to 0$ (up to a further not relabelled subsequence, if needed), there exists $s_t \in [t, t+ e^{-\alpha t}]$ such that $s_n\to s_t$,  $E^{(h_n)}(s_n)\to E(s_t)$ and $F^{(h_n)}(s_n)\to F(s_t)$ in $L^1$. In particular,  the sets $F(s_t)$ (or their complement) are either 
\begin{itemize}
\item[(a')]  union of disjoint disks with equal radii, having total perimeter $P$; 
\item[(b')] or union of disjoint strips (whose area may vary in $t$) with total perimeter bounded by
$ P$.
\end{itemize}

By taking the limit as $h_n \to 0$ in \eqref{eq:expo-discr} we conclude 
\begin{align}\label{eq:mtdec}
  |E(s_t)\triangle F(s_t)| \le C e^{-\frac tC},  
\end{align} 
for some positive constant $C$ independent of $t$.
Hence also the sets $F(s_t)$ converge exponentially fast as $t \to +\infty$ towards $E_\infty$, and thus the sets $F(s_t)$ (or their complement), for every  sufficiently large $t$, must all be of configuration (a') or (b').
This implies that $E_\infty$ (or its complement) is also as in (a') or (b'), which completes the proof of Case 1.

\textbf{Case 2}: There exist $P \in \mathcal{P}_{m,M}$ and $\bar t> 0$ such that $F_\infty=P=f_\infty(t)$ for every $t\geq \bar t$.

Since the functions $f_n$ are monotone, we can conclude that  for every $T>\bar t$ they converge uniformly to $f_\infty\equiv F_\infty$ in $[\bar t, T]$.

From \eqref{eq:thm322}, for every $t\in [\bar t + h_n, T]$, we have
\begin{equation}\label{eq:thm3c2}
\sum_{k=\lfloor\frac{t}{h_n}\rfloor+1}^{\lfloor\frac{T}{h_n}\rfloor}
\frac{h_n}{2}\mathcal{D}(E^{(h_n)}_{k},E^{(h_n)}_{k-1})
\le f_n (t) - f_n (T)=:b_n
\to F_\infty - F_\infty=0 \ \ \ \textup{as } \  h_n\to 0,
\end{equation}
where the convergence is uniform in $t$. 
Combining the above inequality and \eqref{eq:thm312} with the choice $\varepsilon= b_n^{1/4}$, we find 
\begin{align*}
|E^{(h_n)}(t)\Delta E^{(h_n)}(s)|  \leq C b_n^{1/4} +C b_n^{1/4} \to 0 \ \ \ \text{as } \  h_n\to 0,
\end{align*}
for every $\bar t +h_n \leq t < s \leq T$.
Thus, we get $E(t) = E(s)$ for every $\bar t<t<s<T$. Since $T >\bar t$ is arbitrary, we obtain $E(t) =E(s)$ for all $t,s > \bar t.$

Now, we show that, for every $t >\bar t$, the set $E(t)$,  and hence the limiting set $E_\infty$, is the union of disjoint open disks with the same radius or the union of disjoint strip. Reasoning as in Case 1 (cf. \eqref{expo21}-\eqref{expo3}) and using \eqref{eq:thm3c2} we find 
 $s_n \in [t,t+1]$ such that 
\[
\|\kappa_{E^{(h_n)}(s_n)}-\okappa_{E^{(h_n)}(s_n)} \|_{L^2(\bd E^{(h_n)}(s_n))}^2 =o(1).
\]
By Proposition~\ref{prop alex} the sets $E^{(h_n)}(s_n)$ are either small normal deformations of disjoint union of open disks or of parallel strips  with small $C^{1,\frac12}$ norms. 
Arguing as in the previous case, we find that for some $s \in [t,t+1]$, $E(s)$  is either a disjoint union of open disks or of parallel strips. This completes the proof of Case 2. 
\end{proof}

%%%%%%%%%%%%%%%%%%%%%%%%%%%

We now show the convergence of the perimeters along the flow under additional assumptions.
We remind the reader that, unlike in the case of balls, where the radii are the same, strips may have different widths. 
This gives rise to the possibility of a strip converging to a line or two strips merging into one. 
The additional assumptions we require ensure that this cannot happen.
\begin{cor}\label{cor:convper}
Let $\{E(t)\}_{t \ge 0}$ be as in Theorem~\ref{main theorem 2}, and let $\{E^{(h_n)}(t)\}_{t \ge 0}$ be a discrete flow converging to $E(t)$ in $L^1$ as $n \to \infty$. Assume also that $P(E_0)<4$ if $E_\infty$ is as in $(\textnormal{iii})$.
Then, it holds
\begin{equation}\label{eq:exp_dec}
    |P(E(t)) - P(E_\infty)| \leq C e^{-\frac{t}{C}},
\end{equation}
for some positive constant $C$.
\end{cor}
Before proceeding to the proof of the corollary, we prove an analogue of~\cite[Lemma 3.2]{JulMorPonSpa} that is required in our setting. 
\begin{lemma}\label{lem:percon}
    Let $F$ be a union of $d$ disjoint disks of  radius $r$ or a union of parallel strips, with  boundary components having distance at least  $\delta > 0$ one from another.  Then there exists a constant $C$, depending on $d$, $r$ and $\delta$ such that for every set of finite
perimeter $E \subset \mathbb T^2,$ we have 
$$P(F) \le P(E) + C| E \Delta F|.$$
\end{lemma}
\begin{proof}
Since~\cite[Lemma 3.2]{JulMorPonSpa} applies to the case where $F$ is a union of $d$ disjoint disks with radius $r$, we will only prove the result for the case where $F$ is a union of parallel strips. 
Let $\text{sd}_F$ be the signed distance function\footnote{We recall that   the signed distance function to $E$ is defined  as $\sd_E(x) := \dist(x, E) -  \dist(x, \R^2 \setminus E) $ for  $ x \in \R^2.$} from $F$. Let $\zeta$ be a  smooth function such that 
\[
\zeta(x) = \begin{cases}
    1 \quad \text{if } |\text{sd}_F(x)| \leq \delta/8,\\
    0 \quad \text{if } |\text{sd}_F(x)| \geq \delta/4.
\end{cases}
\]
We now define the vector field  $X = \zeta \nabla \text{sd}_F,$  which  satisfies $|\div X| \leq \frac{C}{\delta}.$  Note that 
\[
\Delta \text{sd}_F(x) = 0, \quad  \text{ in } \{|\text{sd}_F| \leq \delta/8\}.
\]  
Thus, by the divergence theorem, we conclude the estimate 
\[
P(F)-P(E) \le \int_{ E \Delta F} |\operatorname{div} X| \le \frac{C}{\delta}|E \Delta F|.
\]  
\end{proof}
\begin{proof}[Proof of Corollary \ref{cor:convper}]
    We consider only the case $(\textnormal{iii})$, since the others have already been proved in \cite{JulMorPonSpa}.
    We note that, under the assumption $P(E_0)<4$, the limit set $E_\infty$ must be a single strip. Hence, in the notation of the proof of Theorem~\ref{main theorem 2}, also the sets $F^{(h_n)}(s_n)$ must be single strips, and also their limit $F(s_t)$ as $n \to \infty$.
    In particular $P(F(s_t))=P(F^{(h_n)}(s_n))=P(E_\infty).$
    
    By Lemma \ref{lem:percon}  with $F=F(s_t)$ and $E=E(t+e^{-\alpha t})$ we get
    \begin{multline*}
    P(F(s_t))\le P(E(t+e^{-\alpha t}))+C|F(s_t)\triangle  E(t +e^{-\alpha t})|\\\le  \liminf_{n\to \infty} P(E^{(h_n)}(t +e^{-\alpha t} ))+C|F(s_t)\triangle E(t+e^{-\alpha t})|.        
    \end{multline*}
    By \eqref{eq:P_est_discr} and monotonicity of the perimeters we estimate
    \[
    P(E^{(h_n)}(t+e^{-\alpha t}))\le P(E^{(h_n)}(s_n)) \le P(F^{(h_n)}(s_n)) +Ce^{-\frac{t}{C}}= P(E_\infty)+Ce^{-\frac{t}{C}}.
    \]
    By combining the above inequalities
    and by taking the limit as $n \to \infty$ we conclude
    \[
    \begin{split}
        P(E_\infty)= P(F(s_t)) \le P(E(t+e^{-\alpha t})) +Ce^{-\frac{t}{C}}
        \le P(E_\infty) +Ce^{-\frac{t}{C}},
    \end{split}
    \]
    which implies the claim.
\end{proof}

\begin{rmrk}\label{rmk:dist_strip}
    Let $E_\infty$ be as in $(\textnormal{iii})$ and
    denote by $\Gamma_i$ the connected components of $\partial E^{(h_n)}(t)$. We note that, if there exists $\delta >0$ such that 
    \begin{equation*}
    \textnormal{dist}_{\Gamma_i} (\partial E^{(h_n)}(t) \setminus \Gamma_i) \ge \delta,
    \end{equation*}
    then we have same conclusion of Corollary \ref{cor:convper}. As already mentioned, this condition is in general nontrivial to prove, as strips may have arbitrary small mass.
\end{rmrk} 

\section{Asymptotic Behavior of the Area-Preserving Curvature Flow}

Let us recall the framework for the construction of the \textit{flat flows} of the area-preserving curvature flow, following~\cite{MSS} and with the insights provided by~\cite{Jul23}. 
Fix an area $m>0$ and a time step $h>0$.  Given a set $E\subseteq \T^2$ with $|E|=m$, we consider the minimization problem 
\beq
\label{def:min-prob}
\min  \Big{\{} P(F) + \frac{1}{h}\int_F  \textnormal{sd}_E \ud x \ :\   |F| = m  \Big{\}}.
\eeq
Note that minimizers exist, but may not be unique.  We define the dissipation of a set $F$ with respect to a set $E$ as
 \beq
\label{def:distance}
\mathcal{D}(F,E) := \int_{F \Delta E} \dist(x,\bd E)\ud x .
\eeq
Observe that the minimization problem \eqref{def:min-prob} can be rewritten as 
\[
\min  \Big{\{} P(F) + \frac{1}{h}\mathcal{D}(F,E)\ :\ 
  |F| = m \Big{\}}.
\]
\textit{Flat flows}, in the spirit~\cite{ATW,LS}, are then defined as follows.
Let $E(0) \subseteq \T^2$ be a  set of finite perimeter, we iteratively set 
\[   E_{k+1}^{(h)}\in \text{argmin}\left\lbrace   P(F) + \frac 1h \mathcal D(F,E_{k}^{(h)} )\ :\ |F|=m\right\rbrace. \]
We define \textit{discrete flow}  $\{E^{(h)}(t)\}_{t \geq 0}$ by setting
\[
E^{(h)}(t) = E_k^{(h)} \qquad \text{for }\, t \in [kh, (k+1) h).
\]
Finally, a \textit{flat flow} solution of the area-preserving curvature flow starting from $E(0)$ is defined as  any  family of sets $\{E(t)\}_{t \geq 0}$ which is a cluster point of $\{E^{(h)}(t)\}_{t \geq 0}$, i.e., 
\[
E^{(h_n)}(t) \to E(t) \quad \text{as } \, h_n \to 0 \quad \text{in } \, L^1 \quad \text{for a.e. }\, t >0 .
\]
Testing the minimality of $E_{k+1}^{(h)}$ with $E_{k}^{(h)}$, we immediately obtain
 \beq
\label{eg:energy-compa}
P(E_{k+1}^{(h)}) + \frac{1}{h}\mathcal{D}(E_{k+1}^{(h)},E_k^{(h)})   \leq P(E_{k}^{(h)}), 
\eeq
in particular, the perimeter is non-increasing along the sequence $E_{k}^{(h)}$. Moreover (see~\cite{Jul23, MSS}), the set $E_{k+1}^{(h)}$ is $C^{2,\alpha}$-regular and satisfies the Euler-Lagrange equation
\beq
\label{eg:Euler-Lag}
\frac{\textnormal{sd}_{E_k^{(h)}}}{h} = - \kappa_{E_{k+1}^{(h)}} + \lambda_{k+1}^{(h)} \qquad \text{on }\, \bd E_{k+1}^{(h)},
\eeq
in the classical sense, where $\lambda_{k+1}^{(h)}$ is the Lagrange multiplier associated to the volume constraint. 
By~\cite[Theorem 1]{Jul23} there exists a \textit{flat flow} starting from $E(0)$ such that $|E(t)| = m$ for every $t \geq 0$ and $P(E(t)) \leq P(E(0))$.

We can now state our main result on the asymptotic behavior of \textit{flat flows} for area-preserving curvature flow.

\begin{thrm}\label{main theorem 1}
    Let $0<m<1$ and let $\{E(t)\}_{t\ge 0}$ be an area-preserving flat flow starting from a set of finite perimeter $E(0)\subseteq \T^2$ with $|E(0)|=m$. Then, the family $\{E(t)\}_{t\ge 0}$ has a unique $L^1$-limit point $E_\infty$, which takes one of the following forms: 
    \begin{itemize}
        \item[(i)] $E_\infty$ is the union of $d\in\N$ disjoint open disks $D_r(x_1),\dots,D_r(x_d)$ with the same area and with $\pi r^2d=m$;
        \item[(ii)] $\T^2\setminus E_\infty$ is the union of $d\in\N$  disjoint open disks $D_r(x_1),\dots,D_r(x_d)$ with the same area and with $\pi r^2d=1-m$;
        \item[(iii)] $E_\infty$ is the union of $\ell\in\N$ disjoint open strips $S_1,\dots, S_\ell$, possibly having different areas, and with $|\bigcup_{i=1}^\ell S_i|=m$.
    \end{itemize}
    In all cases the rate of convergence is exponential, that is
        \begin{equation*}
        | E(t)\triangle E_\infty| \le C e^{-\frac tC},
    \end{equation*}
    for some positive constant $C$ independent of $t$.
\end{thrm}

The proof  of this result is similar to the one of Theorem~\ref{main theorem 2} and~\cite[Theorem 1.2]{JulMorPonSpa}. We just present a sketch of it.
\begin{proof}[\textbf{Proof of Theorem~\ref{main theorem 1}}]

Let $\{E(t)\}_{t\geq0}$ be an \textit{area-preserving flat flow} and let $\{E^{(h_n)}(t)\}_{t\geq0}$ be a \textit{discrete
flow} converging to $E(t)$. 
Set $f_n(t) = P(E^{(h_n)}(t))$. Reasoning as above (cp.~\cite{JulMorPonSpa})  the sequence $\{f_n\}_{n\in\N}$ converges pointwise, up to a subsequence, to some non-increasing function $f_\infty:[0,\infty)\to \R$. 
We set 
\[
F_\infty= \lim_{t\to \infty}f_\infty(t),\quad v^{(h_n)}_t=\frac1{h_n} \textnormal{sd}_{E^{(h_n)}_{ \lfloor t/h_n\rfloor-1}}.
\]
The proof is then divided into two cases, as previously. We detail  just one of the two.  We assume there exist $P<P'$  two consecutive elements of the sequence $ \mathcal{P}_{m,M}$ such that either $P<F_\infty<P'$ or $F_\infty=P$ and $f_\infty(t) >P$ for every $t\in [0,\infty)$.
In particular, there exists $\bar t >0$ such that, for every $T>\bar t$ we can find $\bar n\in \N\setminus\{0\}$ such that
\begin{gather}
\label{expo0.9}
P\leq f_n(t) < P'\qquad \text{and}\qquad P' -f_n(t) \geq \frac{P'-F_\infty}{2}=:\delta_0
\end{gather}
for every $n\geq \bar n$ and   $t \in [ \bar t, T]$.
Reasoning exactly as in the proof of~\cite[Theorem 1.2]{JulMorPonSpa} we get
\beq  \label{expo1.5}
|E(t)\Delta E(s)| \leq  CM  e^{-\frac{t}{4C'_0}} \qquad \text{for all} \, \, \bar t \leq t \leq s\leq t+1
\eeq
showing in particular the exponential convergence in $L^1$ of   $E(t)$  towards a set $E_\infty$, with $|E_\infty| = m$ and $P(E_\infty)\le P(E(0))$. 
 We conclude by showing that $E_\infty$ is the union of disjoint open disks with the same radius or a union of strips, or the complement of such configurations. 

Fix $t\geq \bar t$, $0<\alpha<\frac1{2C'_0}$.
By~\cite[equations (3.9), (3.10)]{JulMorPonSpa}  we have  
\beq\label{expo2}
\begin{split}
&\int_{[t, t+e^{-\alpha t}]} \Big(\int_{\bd E^{(h_n)}(s)} (v^{(h_n)}_s)^2 d\H^1\Big) ds \leq
\frac{1}{h_n} \sum_{i=\lfloor \frac{t}{h_n}\rfloor}^{\lfloor \frac{t+e^{-\alpha t}}{h_n}\rfloor}\int_{\bd E_i^{(h_n)}} \textnormal{sd}_{E_{i-1}^{(h_n)}}^2 \, d \H^1 \\
&\leq C
\sum_{i=\lfloor \frac{t}{h_n}\rfloor}^{\lfloor \frac{t+e^{-\alpha t}}{h_n}\rfloor}\frac{1}{h_n}\mathcal{D}(E^{(h_n)}_{i},E^{(h_n)}_{i-1}) \leq C   e^{-\frac{t}{2C'_0}}\,,
\end{split}
\eeq
for $n$ sufficiently large. 
Now, by \eqref{expo2} and the mean value theorem there exists $s_n \in [t, t+e^{-\alpha t}]$ such that 
\beq\label{MCF:expo3}
\|\kappa_{E^{(h_n)}(s_n)}-\okappa_{E^{(h_n)}(s_n)} \|_{L^2(\bd E^{(h_n)}(s_n))}^2\leq \int_{\bd E^{(h_n)}(s_n)} (v^{(h_n)}_{s_n})^2 d\H^1 \leq Ce^{-\big(\frac{1}{2C'_0}-\alpha\big)t}\,.
\eeq
Here, in the first inequality, we have used \eqref{eg:Euler-Lag}. At this stage, we can proceed as in the proof of Theorem~\ref{main theorem 2}. 
\end{proof}

As for the Mullins-Sekerka case,  we can show an improved convergence of the flow under   additional assumptions (compare Corollary \ref{cor:convper} and Remark \ref{rmk:dist_strip} in the previous section).

\begin{cor}\label{cor:conveper+hausd}
    Let $\{E(t)\}_{t \ge 0}$ be as in Theorem~\ref{main theorem 1}, and let $\{E^{(h_n)}(t)\}_{t \ge 0}$ be a discrete flow converging to $E(t)$ in $L^1$ as $n \to \infty$. Assume also that $P(E_0)<4$ if $E_\infty$ is as in $(\textnormal{iii})$. Then, it holds
    \begin{equation*}
        \sup_{x\in E(t)\triangle E_\infty}\textnormal{dist}(x,\bd E_\infty) + |P(E(t))-P(E_\infty)|\le C e^{-\frac tC},
    \end{equation*}
    for some constant $C>0$ independent of $t$.
\end{cor}

The proof of the corollary is a consequence of Theorem~\ref{main theorem 1} combined with the next lemma in the spirit of \cite[Lemma 4.3]{JN2020}. For details, we refer to   \cite{JulMorOroSpa,JulMorPonSpa}.

\begin{lemma}\label{l:dan}
Let $m,$ $M>0$. Consider $E_0$ a set of  finite perimeter with $|E_0|=m$ and $P(E_0) \leq M$, and  $F$ is a union of disjoint strips. 
For every $\delta>0$ there exist $\e_0,h_0>0$  such that the following holds. 
If for every two distinct components $(\bd F)_i,(\bd F)_j$ of  $\bd F$ it holds  $\max_{(\partial F)_j}\dist_{(\partial F)_i}\ge \delta$, and   
$\{E^{(h)}(t)\}_{t \geq 0 }$ is a discrete flow starting from $E_0$ satisfying
\begin{align}\label{eq:dist_asm}
\sup_{x \in E^{(h)} (t_0)\Delta F} \dist_{\partial F}(x) \leq \e, \qquad \text{for  some }\e \leq \e_0,\  t_0\ge 0,
\end{align}
then there exists a constant $C>0$ (depending on $m, M$ only) such that, if $h\le \min\{h_0,\e\}$, it holds
\[
\sup_{x \in E^{(h)}(t) \Delta F} \dist_{\partial F}(x) \leq C \sqrt\e, \qquad \text{for all } \ t_0 \le t \le t_0 + \e.
\]
\end{lemma}

\begin{proof}
    In this proof we  identify $\T^2$ with $[-\frac 12,\frac12)^2$.
    We fix $h_0,C_1$ as in   \cite[Proposition 4.2]{JN2020} and $\gamma_N$ as in  \cite[Proposition 3.2]{MSS}.  
    Note that in order to estimate $\sup_{E^{(h)}(t) \Delta F} \dist_{\partial F}$ it is sufficient to bound $\max_{\bd E^{(h)}(t)}\dist_{\partial F}.$
    We can assume without loss of generality  that $t_0=0$, $F=S$ is a strip with zero slope (as by the bound on the perimeter we know that there are only a finite number of possible slopes),   and that one component of $\bd S$ is the $x$-axis, denoted by $\Gamma$. We choose $h_0\le \e_0<1$ such that 
    \[
    (C_1+2+\gamma_N) \sqrt{\e_0} \le \delta/4.
    \]
    We start by considering $t=h$.  By the choice of $\e_0$ and  \cite[Proposition 3.2]{MSS} we find that $\max_{(\bd E^{(h)}(h))_i}\dist_{{(\bd E^{(h)}(h))_j}} \ge \delta/2$ for $i\neq j$. Therefore, it is enough to work with a single component of the boundaries of $\bd E^{(h)}(h), \bd E_0 $ and with $\Gamma$. 

    Assume that $\dist_\Gamma(z) > \dist_\Gamma(z')$ for all $z \in \bd E^{(h)}(h)$ and $z' \in \bd E_0$, otherwise there is nothing to prove. 
    Let $\bar z=(\bar x,\bar y)\in \bd E^{(h)}(h)$, and $z'=(x',y')\in\bd E_0$, such that 
    \[
    |\bar y|=\dist_\Gamma(\bar z)=\max_{\bd E^{(h)}(h)}\dist_\Gamma >  |y'|=\dist_\Gamma( z')=\dist_\Gamma (\bd E_0).
    \]
    Let us assume without loss of generality that $\bar y,y'\ge0$, the other cases follow considering only the points of $E^{(h)}(h), E_0$ in $\{y\ge 0\}$ or $\{y\le 0\}$. 
    We have two cases. 
    If $\bar z\notin E_0$ then $\sd_{E_0}(\bar z)\ge 0$, and  it holds 
    \[
       (x, y' +t)\notin E_0\quad \text{for all }t\in (0,\bar y-y'],\ x\in[-1/2,1/2],
    \]
    otherwise we would have $\max_{\bd E_0}\dist_\Gamma\ge\dist_\Gamma((x,y'+t))>y'. $
    In particular, $\sd_{E_0}(\bar z)\ge \bar y-y'.$ 
    Moreover, note that the line  $\{y=\bar y\}$ touches the set $E^{(h)}(h)$ from above, thus $\kappa_{E^{(h)}(h)}(\bar z)\ge 0$. Using the Euler-Lagrange equation \eqref{eg:Euler-Lag} at $\bar z$ we get
    \begin{equation}\label{eq:iterEL}
        h|\lambda^{(h)}(h)|\ge \sd_{E_0}(\bar z) \ge \bar y-y'=\big(\max_{\bd E^{(h)}(h)}\dist_\Gamma\big)- \big (\max_{\bd E_0} \dist_\Gamma\big).
    \end{equation}
    The other case $\bar z\in E_0$ is analogous. 
    We then remark that for $K>0$, \cite[Proposition 4.2]{JN2020} implies 
    \[
    \begin{split}
        \sum_{\ell=1}^K h|\lambda^{(h)}(\ell h)|\le \int_h^{Kh}(\tfrac1{\sqrt{Kh}}+\sqrt{Kh} |\lambda^{(h)}(t)|^2)dt\le \sqrt{Kh}+C_0\sqrt{Kh}.
    \end{split}
    \]
    In particular, letting $K=[\e/h]$ and using \eqref{eq:dist_asm}, \eqref{eq:iterEL} we get
    \[
        \max_{\bd E^{(h)}(h)}\dist_\Gamma\le \max_{\bd E_0} \dist_\Gamma +   h|\lambda^{(h)}(h)| \le \e + \sqrt{Kh}+C_1\sqrt{Kh}\le (C_1+2)\sqrt\e\le \delta/4,
    \]
    where the last inequality follows from the choice of $\e_0.$

    Using \cite[Proposition 3.2]{MSS}, we note that the set $E^{(h)}(2h)$ still satisfies the starting assumption $\max_{(\bd E^{(h)}(2h))_i}\dist_{{(\bd E^{(h)}(2h))_j}} \ge \delta/2$ for $i\neq j$. We can reason as before to get 
    \[
        \max_{\bd E^{(h)}(2h)}\dist_\Gamma\le \max_{\bd E^{(h)}(h)} \dist_\Gamma +   h|\lambda^{(h)}(2h)| \le  \max_{\bd E_0} \dist_\Gamma +   \sum_{\ell =1}^2 h|\lambda^{(h)}(\ell h)|\le \delta/4.
    \]
    Therefore, the procedure can be iterated and this concludes the proof. 
\end{proof}

\medskip
%%%%%%%%%%%%%%%%%%%%%%%%%%
\section*{Acknowledgements}
The authors wish to thank professors V. Julin and M. Morini for helpful discussions and comments. V. Arya was 
supported by the Academy of Finland grant 314227.  
D. De Gennaro is member of the Gruppo Nazionale per l’Analisi Matematica, la Probabilità e le loro Applicazioni (GNAMPA) of the Istituto Nazionale di Alta Matematica (INdAM).  
D. De Gennaro was funded by the European Union: the European Research Council (ERC), through StG ``ANGEVA'', project number: 101076411. Views and opinions expressed are however those of the authors only and do not necessarily reflect those of the European Union or the European Research Council. Neither the European Union nor the granting authority can be held responsible for them.
A. Kubin research has been supported by the Austrian Science Fund (FWF) through grants 10.55776/F65, 10.55776/P35359, 10.55776/Y1292.

\printbibliography 
\end{document}